\documentclass[english]{amsart}

\usepackage{amsmath}
\usepackage{amssymb}
\usepackage{amsthm}
\usepackage{amscd}
\usepackage{babel}
\usepackage[latin1]{inputenc}
\usepackage{graphicx}

\newtheorem{teor}{Theorem}
\newtheorem{cor}{Corollary}
\newtheorem{prop}{Proposition}
\newtheorem{lem}{Lemma}

\newtheorem{rem}{Remark}

\author{Jos\'{e} Mar\'{i}a Grau}
\address{Departamento de Matemáticas, Universidad de Oviedo\\ Avda. Calvo Sotelo, s/n, 33007 Oviedo, Spain}
\email{grau@uniovi.es}

\author{Antonio M. Oller-Marc\'{e}n}
\address{Centro Universitario de la Defensa\\ Ctra. de Huesca, s/n, 50090 Zaragoza, Spain}
\email{oller@unizar.es}

\title{Fast computation of the number of solutions to $x_1^2+\cdots+x_k^2 \equiv \lambda \pmod{n}$}
\begin{document}
\maketitle
\begin{abstract}
In this paper we study the multiplicative function $\rho_{k,\lambda}(n)$ that counts the number of incongruent solutions of the equation $x_1^2+\cdots+x_k^2 \equiv \lambda\pmod{n}$. In particular we give closed explicit formulas for $\rho_{k,\lambda}(p^s)$ with a arithmetic complexity of constant order.
\end{abstract}

\section{Introduction}

Let $k$, $\lambda$ and $n$ be positive integers and let $\rho_{k,\lambda}(n)$ denote the number of incongruent solutions of the equation
$$x_1^2+x_2^2+\cdots +x_k^2\equiv \lambda \pmod {n}.$$
In other terms:
$$\rho_{k,\lambda}(n):=\textrm{card}\ \{(x_i,\dots,x_k)\in(\mathbb{Z}/n\mathbb{Z})^k : x_1^2+\cdots +x_k^2\equiv\lambda\pmod{n}\}$$

Since the function $\rho_{k,\lambda}$ is multiplicative, it is enough to consider the case when $n=p^s$ is a prime power. Moreover, it is also clear that we can introduce the restriction $0\leq \lambda <n$.

The computation of $\rho_{k,\lambda}(n)$ by mere exhaustive search is obviously inefficient since its computational complexity has order $\Theta(n^k)$. Thus, the interest to find closed formulas involving a number of  operations
which is as small as possible.

Identities for $\rho_{k,\lambda}(n)$ can be derived using Gauss and Jacobi sums. In fact, we have (see \cite{Koro}) a very compact expression like:

\begin{equation}\label{gaus}
   \rho_{k,\lambda}(n)={1\over n} \sum_{ a=1}^ne^{-{2\pi i\frac{a\lambda}{n}}}\left (\sum_{x=1}^ne^{2\pi i{ a x^2\over n}}\right )^k.
\end{equation}

This expression has theoretical value and it could even be practically applied for small values of $n$. Nevertheless, it is not useful for moderately big values of $n$, even in the particularly simple case $\lambda=0$. This is because the arithmetic complexity of that formula is $\Theta(n^2)$.

Another compact expression can be found in \cite{toth}. Namely,

 \begin{equation}\label{dos}
    \rho_{k,\lambda}(n)=n^{k-1} \sum_{d \mid n} \frac{1}{d^k} \sum_{\substack{l=1\\ \gcd(l,r)=1}}^d e^{\frac{-2\pi i l \lambda}{d}}   S(l,d)^k ,
\end{equation}

where $S(l,r)$ is the quadratic Gauss sum defined by
$$S(l,r):=\sum_{\substack{j=1\\ \gcd(l,r)=1}}^r\exp(2 \pi i l j^2/r).$$
This formula is also inefficient, even in the prime-power case. In fact, if $n=p^s$ the arithmetic complexity is $\Theta(s^3)$.

Some efficient explicit formulas are known for some particular cases. For instance, V.H. Lebesgue \cite{jor2} gave in 1837 a closed formula for $\rho_{k,\lambda}(p)$. In \cite[p. 46]{cero} a formula for $\rho_{k,0}(p^s)$ is given and the case $\gcd(\lambda,p)=1$ was completely solved in \cite{concat} giving. Finally, in \cite{toth} and \cite{Cal} we can find closed formulas for some particular cases of $k$ and $\lambda$.

Nevertheless, up to date, no general formula with constant (independent of $k$, $\lambda$ and $s$) complexity for the computation of $\rho_{k,\lambda}(p^s)$ has been given. Thus, with the results that are known today it is not possible to compute (in a reasonable time) the value of $\rho_{10,5^{100000}}(5^{1000000})$, for instance.

In this work, we present explicit general formulas for $\rho_{k,\lambda}(p^s)$ with arithmetic complexity of constant order, $O(1)$. We use elementary techniques that do not involve Gauss or Jacobi sums.

\section{Known basic cases}

The formulas for $\rho_{k,\lambda}(p^s)$ that we are going to present ultimately rely on the values of $\rho_{k,\lambda}(p)$ if $p$ is an odd prime and on the values of $\rho_{k,\lambda}(2^s)$ with $1\leq s\leq 3$ if $p=2$.

As we already pointed out, when $p$ is an odd prime the values of $\rho_{k,\lambda}(p)$ were already studied by V.H. Lebesgue in 1837. In particular he proved the following result \cite[Chapter X]{jor2}, where $\left(\frac{\lambda}{p}\right)$ denotes the Legendre symbol defined by
$$\left(\frac{\lambda}{p}\right)=\begin{cases} 1, & \textrm{if $\lambda$ is a quadratic residue modulo $p$};\\-1, & \textrm{if $\lambda$ is a not a quadratic residue modulo $p$};\\0, & \textrm{if $p\mid\lambda$}.\end{cases}$$

\begin{prop}
\label{LEB}
Let $p$ be an odd prime and let $k,\lambda$ be positive integers with $0\leq\lambda<p$. Put $t=(-1)^{(p-1)(k-1)/4}p^{(k-1)/2}$ and $l=(-1)^{k(p-1)/4}p^{(k-2)/2}$. Then,
$$\rho_{k,\lambda}(p)=\begin{cases} p^{k-1}+\left(\frac{\lambda}{p}\right)t, & \textrm{If $k$ is odd};\\p^{k-1}-l+\left(1-\left|\left(\frac{\lambda}{p}\right)\right|\right)pl, & \textrm{If $k$ is even}.
\end{cases}$$
\end{prop}

In the $p=2$ case, formulas for $\rho_{k,\lambda}(2^s)$ with $1\leq s\leq 3$ were given in \cite{concat} when $\lambda$ is even. Here we complete it.

\begin{prop}
\label{PROP2}
Let $k$ be a positive integer. Then:
\begin{itemize}
\item[i)] $\rho_{k,1}(2)=\rho_{k,0}(2)=2^{k-1}$,
\item[ii)] $\rho_{k,0}(4)=4^{-1 + k} + 2^{-1 + \frac{3\,k}{2}}\,\cos (\frac{k\,\pi }{4})$,
\item[iii)] $\rho_{k,1}(4) = 4^{k-1}+2^{\frac{3 k}{2}-1} \sin \left(\frac{\pi  k}{4}\right)$,
\item[iv)] $\rho_{k,2}(4)=4^{-1 + k} - 2^{-1 + \frac{3\,k}{2}}\,\cos (\frac{k\,\pi }{4})$
\item[v)] $\rho_{k,3}(4) = 4^{k-1}-2^{\frac{3 k}{2}-1} \sin \left(\frac{\pi  k}{4}\right)$,
\item[vi)] $\rho_{k,0}(8)=8^{-1 + k} + 2^{-2 + 2\,k}\,\cos (\frac{k\,\pi }{4}) + 2^{-2 + \frac{5\,k}{2}}\,\cos (\frac{k\,\pi }{4}) + 2^{-2 + 2\,k}\,\cos (\frac{3\,k\,\pi }{4})$
\item[vii)] $\rho_{k,1}(8)=  2^{2 k-3} \left(2^k+2^{\frac{k}{2}+1} \sin \left(\frac{\pi  k}{4}\right)+2 \sin \left(\frac{1}{4} \pi(k+1)\right)-2 \cos \left(\frac{1}{4} (3 \pi  k+\pi )\right)\right)$,
\item[viii)] $ \rho_{k,2}(8)=8^{-1 + k} - 2^{-2 + \frac{5\,k}{2}}\,\cos (\frac{k\,\pi }{4}) + 2^{-2 + 2\,k}\,\sin (\frac{k\,\pi }{4}) - 2^{-2 + 2\,k}\,\sin (\frac{3\,k\,\pi }{4})$
\item[ix)] $\rho_{k,3}(8)=2^{2 k-3} \left(2^k-2^{\frac{k}{2}+1} \sin \left(\frac{\pi  k}{4}\right)-2 \left(\cos \left(\frac{1}{4} \pi (k+1)\right)+\cos \left(\frac{3}{4} \pi  (k+1)\right)\right)\right)$,
\item[x)] $\rho_{k,4}(8)=8^{-1 + k} - 2^{-2 + 2\,k}\,\cos (\frac{k\,\pi }{4}) + 2^{-2 + \frac{5\,k}{2}}\,\cos (\frac{k\,\pi }{4}) - 2^{-2 + 2\,k}\,\cos (\frac{3\,k\,\pi }{4})$
\item[xi] $\rho_{k,5}(8)=2^{2 k-3} \left(2^k+2^{\frac{k}{2}+1} \sin \left(\frac{\pi  k}{4}\right)-2 \sin \left(\frac{1}{4} \pi (k+1)\right)+2 \cos \left(\frac{1}{4} (3 \pi  k+\pi )\right)\right)$,
\item[xii)] $\rho_{k,6}(8)=8^{-1 + k} - 2^{-2 + \frac{5\,k}{2}}\,\cos (\frac{k\,\pi }{4})-2^{-2 + 2\,k}\,\sin (\frac{k\,\pi }{4})+2^{-2 + 2\,k}\,\sin (\frac{3\,k\,\pi }{4})$,
\item[xiii)] $\rho_{k,7}(8)= 2^{2 k-3} \left(2^k-2^{\frac{k}{2}+1} \sin \left(\frac{\pi  k}{4}\right)-2 \sin \left(\frac{1}{4} (3 \pi  k+\pi)\right)+2 \cos \left(\frac{1}{4} \pi  (k+1)\right)\right)$.
\end{itemize}
\end{prop}
\begin{proof}
Given $k,n\in \mathbb{N}$, let us define the matrix $M(n)=\left(\rho_{1,i-j}(n)\right)_{0\leq i,j\leq n-1}$. If we consider the column vector
$R_{k}(n)=\left(\rho_{k,i}(n) \right)_{0\leq i\leq n-1}$, the following recurrence relation holds:
$$R_{k}(n)=M(n)\cdot R_{k-1}(n).$$
Then, it is enough to apply elementary linear algebra techniques. For details, see \cite[Lemma 4]{concat}.
\end{proof}

\section{Preparatory results}

Given positive integers $k,n$ and $0\leq \lambda <n$, let $A(k,\lambda,n)$ denote
the set of solutions $(x_1,\dots,x_k)\in(\mathbb{Z}/n\mathbb{Z})^k$ of the congruence $x_1^2+\cdots +x_k^2\equiv \lambda\pmod{n}$. In particular, if $n=p^s$ is a prime-power, we have that
$$A(k,\lambda,p^s)=\{(x_1,...,x_k) \in (\mathbb{Z}_{p^s})^k :  x_1^2+\cdot\cdot\cdot+x_k^2 \equiv \lambda \pmod {p^s}\}.$$
Now, in this situation, let us define the following sets:
$$A_1(k,\lambda,p^s)=\{(x_1,...,x_k) \in A(k,\lambda,p^s) :  p\nmid x_i\ \textrm{for some $1\leq i\leq k$}\},$$
$$A_2(k,\lambda,p^s)=\{(x_1,...,x_k) \in A(k,\lambda,p^s) :  p\mid x_i\ \textrm{for every $1\leq i\leq k$}\}.$$
Note that $A(k,\lambda,p^s)=A_1(k,\lambda,p^s)\cup A_2(k,\lambda,p^s$. Hence, since $A_1(k,\lambda,p^s)$ and $A_2(k,\lambda,p^s$ are disjoint, if we define $\rho^{(1)}_{k,\lambda}(p^s):=\textrm{card}(A_1(k,\lambda,p^s))$ and $\rho^{(2)}_{k,\lambda}(p^s):=\textrm{card}(A_2(k,\lambda,p^s))$ it follows that
$$\rho_{k,\lambda}(p^s)=\rho^{(1)}_{k,\lambda}(p^s)+\rho^{(2)}_{k,\lambda}(p^s).$$

\begin{rem}
If $\gcd(\lambda,p)=1$; i.e., if $p\nmid\lambda$ then $A_2(k,\lambda,p^s)=\emptyset$. Thus, $\rho^{(2)}_{k,\lambda}(p^s)=0$ and it follows that $\rho^{(1)}_{k,\lambda}(p^s)=\rho_{k,\lambda}(p^s)$.
\end{rem}

This remark implies that the proof of the following result is the same as that of Lemmata 1 and 2 in \cite{concat}.

\begin{prop}
\ \label{rec1}
\begin{itemize}
\item[i)] Let $p^s$ be an odd prime-power with $s\geq 1$ and $0\leq \lambda <p^s$. Then, $$\rho^{(1)}_{k,\lambda}(p^s)=p^{(s-1)(k-1)} \rho^{(1)}_{k,\lambda}(p).$$
\item[ii)] Let $s\geq 3$ and $0\leq \lambda<2^s$. Then, $\rho_{k,\lambda}^{(1)}(2^{s})=2^{(s-3)(k-1)}\rho_{k,\lambda}^{(1)}(8)$.
\end{itemize}
\end{prop}

Proposition \ref{rec1} provides us with a recursive relation for $\rho^{(1)}_{k,\lambda}(p^s)$. Note that this result implies that we will have to study the case $p=2$ separately.

Now we turn to $\rho^{(2)}_{k,\lambda}(p^s)$. In this case, we have the following result.

\begin{prop}\label{rec2}
Let $p^s$ be a prime-power, with $s\geq 1$ and let $0\leq \lambda<p^s$. Then,
$$\rho^{(2)}_{k,\lambda}(p^s)=\begin{cases}1, & \textrm{if $s=1$ and $\lambda=0$};\\ p^k, & \textrm{if $s=2$ and $\lambda=0$};\\ p^k\rho_{k,\lambda/p^2}(p^{s-2}), & \textrm{if $s\geq 3$ and $p^2\mid\lambda$};\\ 0, & \textrm{otherwise}.\end{cases}$$
\end{prop}
\begin{proof}
If $s=1$ and $\lambda=0$, it is obvious that the only $k$-tuple $(x_1,\dots,x_k)$ such that $x_1^2+\cdots +x_k^2\equiv 0\pmod{p}$ and $p\mid x_i$ for every $i$ is $(0,\dots, 0)$. Hence, $\rho^{(2)}_{k,\lambda}(p^s)=1$ in this case.

Secondly, if $s=2$ and $\lambda=0$, $\rho^{(2)}_{k,\lambda}(p^2)=\rho^{(2)}_{k,0}(p^2)$ is the number of $k$-tuples $x_1,\dots,x_k)$ such that $x_1^2+\cdots +x_k^2\equiv 0\pmod{p^2}$ and $p\mid x_i$ for every $i$. It is obvious that there are $p^k$ such $k$-tuples because $x_i$ can be any multiple of $p$ in $\mathbb{Z}/p^2\mathbb{Z}$.

Now, assume that $p^2\mid\lambda$ and $s\geq 3$. First of all, using Euclid's algorithm, it is easy to see that every element of $A_2(k,\lambda,p^s)$ can be written in the form $(px_1+\alpha_1p^{s-1},\dots,px_k+\alpha_kp^{s-1})$ with $0\leq \alpha_i\leq p-1$ and $(x_1\dots,x_k)\in A(k,\lambda/p^2,p^{s-2})$.

On the other hand, let $(x_1,\dots,x_k)\in A(k,\lambda/p^2,p^{s-2})$; i.e., $x_1^2+\cdots +x_k^2\equiv \lambda/p^2\pmod{p^{s-2}}$. Clearly the set
$$\{(px_1+\alpha_1p^{s-1},\dots,px_k+\alpha_kp^{s-1}) : 0\leq \alpha_i\leq p-1\ \textrm{for every $1\leq i\leq k$}\}$$
is contained in $A_2(k,\lambda,p^s)$ because
$$(px_1+\alpha_1p^{s-1})^2+\cdots +(px_k+\alpha_kp^{s-1})^2\equiv p^2(x_1^2+\cdots+ x_k^2)\equiv \lambda\pmod{p^s}$$
and all its elements are incongruent modulo $p^s$. Thus, every element of the set $A(k,\lambda/p^2,p^{s-2})$ gives rise to $p^k$ different elements of $A_2(k,\lambda,p^s)$ and the result follows.

Finally, in the remaining cases (i.e., if $s=1$ or $2$ with $0<\lambda<p^2$ or if $s\geq 3$ with $p^2\nmid\lambda$) it is obvious that $A_2(k,\lambda,p^s)=\emptyset$ and hence $\rho^{(2)}_{k,\lambda}(p^s)=0$, as claimed.
\end{proof}

With the help of Proposition \ref{rec1} and Proposition \ref{rec2} we can give recursive formulas that express the value of $\rho_{k,\lambda}(p^s)$. First, we deal with the odd $p$ and non-zero $\lambda$ case.

\begin{teor} \label{oddnz}
Let $p^s$ be an odd prime-power and let $0<\lambda<p^s$ be an integer. Put $\lambda=p^r\lambda'$ with $0\leq r<s$ and $p\nmid\lambda'$. Then,
$$\rho_{k,\lambda}(p^s)=\sum_{i=0}^{\lfloor r/2\rfloor}   p^{ki+(s-2i-1)(k-1)}\cdot\rho_{k,\lambda/p^{2i}}^{(1)}(p).$$
\end{teor}
\begin{proof}
We have that $\rho_{k,\lambda}(p^s)=\rho^{(1)}_{k,\lambda}(p^s)+\rho^{(2)}_{k,\lambda}(p^s)$. If $r\leq 1$, then Proposition \ref{rec2} implies that $\rho^{(2)}_{k,\lambda}(p^s)=0$. Hence,
$\rho_{k,\lambda}(p^s)=\rho^{(1)}_{k,\lambda}(p^s)=p^{(s-1)(k-1)}\rho^{(1)}_{k,\lambda}(p)$ due to Proposition \ref{rec1} and we are done.

Now, if $r\geq 2$ the $s\geq 3$ and Proposition \ref{rec2} implies that
$$\rho^{(2)}_{k,\lambda}(p^s)=p^k\rho_{k,\lambda/p^2}(p^{s-2})=p^k\rho^{(1)}_{k,\lambda/p^2}(p^{s-2})+p^k\rho^{(2)}_{k,\lambda/p^2}(p^{s-2}).$$
Thus, using Proposition \ref{rec1} again we obtain that
$$\rho_{k,\lambda}(p^s)=p^{(s-1)(k-1)}\rho^{(1)}_{k,\lambda}(p)+p^k p^{(s-3)(k-1)}\rho^{(1)}_{k,\lambda/p^2}(p)+p^k\rho^{(2)}_{k,\lambda/p^2}(p^{s-2}).$$
Since $\lambda/p^2=p^{r-2}\lambda'$, if $r-2\leq 1$, then $\rho^{(2)}_{k,\lambda/p^2}(p^{s-2})=0$ by Proposition \ref{rec2} and we are done.

If, on the other hand, $r\geq 4$ then $s-2\geq 3$ and Proposition \ref{rec2} implies that $\rho^{(2)}_{k,\lambda/p^2}(p^{s-2})=p^k\rho^{(2)}_{k,\lambda/p^4}(p^{s-4})$. Thus, using Proposition \ref{rec1} again, it follows that
$$\rho_{k,\lambda}(p^s)=\sum_{i=0}^{2} \Big(p^{ki+(s-2i-1)(k-1)}\rho_{k,\lambda/p^{2i}}^{(1)}(p)\Big)+p^{2k}\rho^{(2)}_{k,\lambda/p^4}(p^{s-4}).$$

Clearly this process can be iteratively repeated until we reach the expression
$$\rho_{k,\lambda}(p^s)=\sum_{i=0}^{\lfloor r/2\rfloor} \Big(p^{ki+(s-2i-1)(k-1)}\cdot\rho_{k,\lambda/p^{2i}}^{(1)}(p)\Big) + p^{k\lfloor r/2\rfloor} \rho^{(2)}_{k,\lambda/p^{2\lfloor r/2\rfloor}}(p^{s-2\lfloor r/2\rfloor})$$
and, since $p^2\nmid \lambda/p^{2\lfloor r/2\rfloor}$ the result follows from Proposition \ref{rec2}.
\end{proof}

Now, we turn to the $\lambda=0$ case for an odd prime $p$.

\begin{teor} \label{oddz}
Let $p^s$ be an odd prime-power. Then,
$$\rho_{k,0}(p^s)=\sum_{i=0}^{\lfloor (s-1)/2\rfloor} \Big(p^{ki}p^{(s-2i-1)(k-1)}\cdot \rho^{(1)}_{k,p^{s-2i}}(p)\Big) + p^{\lfloor s/2\rfloor k}.$$
\end{teor}
\begin{proof}
First of all, note that $\rho_{k,0}(p^s)=\rho_{k,p^s}(p^s)$. Then we can proceed recursively just like in Theorem \ref{oddnz} because $\rho_{k,p^s}(p^s)=\rho^{(1)}_{k,p^s}(p^s)+\rho^{(2)}_{k,p^s}(p^s)$.

If $s=1$, then $\rho_{k,p}(p)=\rho^{(1)}_{k,p}(p)+\rho^{(2)}_{k,p}(p)=\rho^{(1)}_{k,p}(p)+1$ due to Proposition \ref{rec2}.

If $s=2$, $\rho_{k,p^2}(p^2)=\rho^{(1)}_{k,p^2}(p^2)+\rho^{(2)}_{k,p^2}(p^2)=p^{k-1}\rho^{(1)}_{k,p^2}(p)+p^k$ due to Propositions \ref{rec1} and \ref{rec2}.

Now, if $s\geq 3$, then Propositions \ref{rec1} and \ref{rec2} imply that
$$\rho_{k,p^s}(p^s)=\rho^{(1)}_{k,p^s}(p^s)+\rho^{(2)}_{k,p^s}(p^s)=p^{(s-1)(k-1)}\rho_{k,p^s}(p)+p^k\rho_{k,p^{s-2}}(p^{s-2})$$
If $s-2=1$ or $s-2=2$, then we apply Proposition \ref{rec2} and the result follows. If, on the other hand, $s-2\geq 3$ then
$$\rho_{k,p^s}(p^s)=p^{(s-1)(k-1)}\rho_{k,p^s}(p)+p^k\Big(\rho^{(1)}_{k,p^{s-2}}(p^{s-2})+\rho^{(2)}_{k,p^{s-2}}(p^{s-2})$$
so applying Propositions \ref{rec1} and \ref{rec2} again we get that
$$\rho_{k,p^s}(p^s)=p^{(s-1)(k-1)}\rho_{k,p^s}(p)+p^kp^{(s-3)(k-1)}\rho_{k,p^{s-2}}(p)+p^{k}\rho_{k,p^{s-4}}(p^{s-4}).$$

To conclude the proof it is enough to observe that the previous process will end after $\lfloor (s-1)/2\rfloor$ steps and hence after $\lfloor (s-1)/2\rfloor + 1$ applications of Propositions \ref{rec1} and \ref{rec2}.
\end{proof}

Now, for the case $p=2$ and non-zero $\lambda$ we have the following result.

\begin{teor} \label{2nz}
Let $2^s$ be a power of two ($s\geq 1$) and let $0<\lambda<2^s$ be an integer. Put $\lambda=2^r\lambda'$ with $0\leq r<s$ and odd $\lambda'$.
Then,
$$\rho_{k,\lambda}(2^s)=\sum_{i=0}^{\lfloor
\frac{r}{2}\rfloor-1}\Big(2^{ki}2^{(s-2i-3)(k-1)}\cdot\rho^{(1)}_{k,\lambda/2^{2i}}(8)
\Big) +
2^{k\lfloor\frac{r}{2}\rfloor}\rho^{(1)}_{k,\lambda/2^{2\lfloor\frac{r}{2}\rfloor}}(2^{s-2\lfloor\frac{r}{2}\rfloor}).$$
\end{teor}
\begin{proof}
The proof goes exactly as in Theorem \ref{oddnz} using Proposition \ref{rec1} and Proposition \ref{rec2} repeatedly. Note that, in the cases $r=0$ and $r=1$ we consider that if the upper summation limit is $-1$, the sum is empty.
\end{proof}

And finally, the case $p=2$, and $\lambda=0$ is given by the following result.

\begin{teor} \label{2z}
Let $2^s$ be a power of two ($s\geq 1$). Then,

$$\rho_{k,0}(2^s)=\sum_{i=0}^{\lfloor \frac{s-1}{2}\rfloor-1} \Big(2^{ki}2^{(s-2i-3)(k-1)}\cdot \rho^{(1)}_{k,2^{s-2i}}(8)\Big) + 2^{\lfloor\frac{s-1}{2}\rfloor
k}\cdot\rho^{(1)}_{k,0}(2^{s-2\lfloor\frac{s-1}{2}\rfloor})
+2^{\lfloor \frac{s}{2}\rfloor k}.$$
\end{teor}
\begin{proof}
The proof goes exactly as in Theorem \ref{oddz} using Proposition \ref{rec1} and Proposition \ref{rec2} repeatedly. Note that, in the cases $s=1$ and $s=2$ we consider that if the upper summation limit is $-1$, the sum is empty.
\end{proof}

\section{Fast computation of $\rho_{k,\lambda}(p^s)$}

With the results that we have proved in the previous section, we have a procedure to compute $\rho_{k,\lambda}(p^s)$ which has arithmetic complexity of order $O(s)$. Nevertheless, as we are going to see in this section, it is possible to obtain formulas requiring a constant number of operations.

To do so, given integer numbers $k$, $p$, $s$ and $N$, we define the function
$$\Omega(k,p,s,N):=\sum_{i=0}^{N}   p^{ki+(s-2i-1)(k-1)}.$$
Since it is essentially a geometric series, the following result is straightforward.

\begin{lem}\label{L2}
Let $k$, $p$, $s$ and $N$ be integer numbers. Then,
$$\Omega(k,p,s,N)=\begin{cases}\frac{-1 + p^{1 + N}}{-1 + p}, & \textrm{if $k=1$};\\   p^{-1 + s} (1+N), & \textrm{if $k=2$};\\
  \frac{p^{\left( -1 + k \right) \,\left( -1 + s \right)}
    \left( p^k - p^2\,{\left( p^{2 - k} \right) }^N \right) }
    {-p^2 + p^k}, & \textrm{otherwise}.\end{cases}$$
\end{lem}

The following result will also be useful in the sequel.

\begin{lem}\label{L1}
\
\begin{itemize}
\item[i)] Let $p$ be any prime and let $0\leq\lambda <p$. Then,
$$\rho^{(1)}_{k,\lambda}(p)=\begin{cases} \rho_{k,\lambda}(p)-1, & \textrm{if $\lambda=0$};\\ \rho_{k,\lambda}(p), & \textrm{if $\lambda\neq0$}.\end{cases}$$
\item[ii)] Let $0\leq\lambda <4$. Then,
$$\rho^{(1)}_{k,\lambda}(4)=\begin{cases} \rho_{k,\lambda}(4)-2^k, & \textrm{if $\lambda=0$};\\ \rho_{k,\lambda}(4), & \textrm{if $\lambda\neq0$}.\end{cases}$$
\item[iii)] Let $0\leq\lambda <8$. Then,
$$\rho^{(1)}_{k,\lambda}(8)=\begin{cases} \rho_{k,\lambda}(8)-2^{2k-1}, & \textrm{if $\lambda=0,4$};\\ \rho_{k,\lambda}(8), & \textrm{if $\lambda\neq 0,4$}.\end{cases}$$
\end{itemize}
\end{lem}
\begin{proof}
Just recall that $\rho^{(1)}_{k,\lambda}(n)=\rho_{k,\lambda}(n)-\rho^{(2)}_{k,\lambda}(n)$ and apply Proposition \ref{rec2}.
\end{proof}

\begin{cor}
Let $p^s$ be an odd prime-power and let $0<\lambda<p^s$ be an integer. Put $\lambda=p^r\lambda'$ with $0\leq r<s$ and $p\nmid\lambda'$. Then,
$$\rho_{k,\lambda}(p^s)=\begin{cases}\Omega(k,p,s,\frac{ r-1}{2}) \cdot(\rho_{k,0}^{ }(p)-1), & \textrm{if $r$ is odd};\\ \Omega(k,p,s,\frac{r-2}{2})\cdot(\rho_{k,0}^{ }(p)-1)+p^{k\frac{r }{2}+(s-r-1)(k-1)}\cdot\rho_{k,\lambda'}^{}(p), & \textrm{if $r$ is even }.\end{cases}$$
\end{cor}
\begin{proof}
Using Lemma \ref{L1} the following hold:
\begin{itemize}
\item If $r$ is odd, then for every $i\leq\lfloor r/2\rfloor$ we have that $\rho_{k,\lambda/p^{2i}}^{(1)}(p)=\rho_{k,0}^{(1)}(p)=\rho_{k,0}^{ }(p)-1$.
\item On the other hand, if $r$ is even then $\rho_{k,\lambda/p^{2i}}^{(1)}(p)=\rho_{k,0}^{(1)}(p)=\rho_{k,0}^{ }(p)-1$ for every $i<r/2$, while $\rho_{k,\lambda/p^{2i}}^{(1)}(p)=\rho_{k,\lambda'}^{}(p)$ for $i=r/2$.
\end{itemize}
Hence, from Theorem \ref{oddnz} it follows that
$$\rho_{k,\lambda}(p^s)=\sum_{i=0}^{  r/2-1}   p^{ki+(s-2i-1)(k-1)}\cdot(\rho_{k,0}^{ }(p)-1)+ p^{kr/2+(s-r-1)(k-1)}\cdot\rho_{k,\lambda'}^{}(p)$$
and Lemma \ref{L2} concludes the proof.
\end{proof}

\begin{cor}
Let $p^s$ be an odd prime-power. Then,
$$\rho^{ }_{k,0}(p^s)=\Omega(k,p,s,\lfloor \frac{ s-1}{2}\rfloor) \cdot(\rho_{k,0}^{ }(p)-1) +p^{k\lfloor s/2\rfloor  }$$
\end{cor}
\begin{proof}
First, observe that $s-2i >0$ for every $i\leq \lfloor \frac{s-1}{2}\rfloor$. Thus, Lemma \ref{L1} implies that $\rho^{(1)}_{k,p^{s-2i}}(p)=\rho_{k,0}(p)-1$. Consequently, it is enough to apply theorem \ref{oddz} to get that
$$\rho_{k,0}(p^s)=  (\rho^{ }_{k,0}(p)-1) \cdot\sum_{i=0}^{\lfloor (s-1)/2\rfloor} \Big(p^{ki}p^{(s-2i-1)(k-1)}\Big) + p^{\lfloor s/2\rfloor k}$$
and the result follows.
\end{proof}

\begin{cor}
Let $2^s$ be a power of two ($s\geq 3$) and let $0<\lambda<2^s$ be an integer. Put $\lambda=2^r\lambda'$ with $0\leq r<s$ and odd $\lambda'$.
Then,
\begin{itemize}
\item[i)] If $r$ is odd and $s-r>1$,
$$\rho^{ }_{k,\lambda}(2^s)=\frac{\Omega(k,2,s,\frac{ r-3}{2})}{2^{2(k-1)} } \cdot(\rho^{ }_{k,0}(8)- 2^{2 k -1})
+2^{k \frac{r-1}{2}+(s-r-2)(k-1) }\rho^{}_{k,\lambda' 2   }(8).$$
\item[ii)] If $r$ is odd and $s-r=1$,
$$\rho^{ }_{k,\lambda}(2^s)=\frac{\Omega(k,2,s,\frac{ r-3}{2})}{2^{2(k-1)} } \cdot(\rho^{ }_{k,0}(8)- 2^{2 k -1})
+2^{k \frac{r-1}{2} }\rho^{ }_{k,2\lambda' }(4).$$
\item[iii)] If $r$ is even and $s-r>2$,
\begin{equation*}\begin{split} \rho_{k,\lambda}(2^s)&= \frac{1}{2^{2(k-1)}} \Omega(k,2,s,\frac{ r-4}{2})\cdot(\rho^{ }_{k,0}(8)-2^{2k-1})+\\
&\ \  +2^{1 + r - s + k\,\left( -2 - \frac{r}{2} + s \right) }  (\rho^{ }_{k,4\lambda'}(8)-2^{2k-1})+
2^{k \frac{r}{2}+(s-r-3)(k-1) }\rho^{ }_{k,\lambda'}(8).
     \end{split} \end{equation*}
\item[iv)] If $r$ is even and $s-r=2$,
\begin{equation*}\begin{split} \rho_{k,\lambda}(2^s)&= \frac{1}{2^{2(k-1)}} \Omega(k,2,s,\frac{ r-4}{2})\cdot(\rho^{ }_{k,0}(8)-2^{2k-1})+\\
&\ \  +2^{1 + r - s + k\,\left( -2 - \frac{r}{2} + s \right) }  (\rho^{ }_{k,4\lambda'}(8)-2^{2k-1})+
2^{k \frac{r}{2} }\rho^{ }_{k,\lambda'}(4).
     \end{split} \end{equation*}		
\item[v)] If $r$ is even and $s-r=1$,
\begin{equation*}\begin{split} \rho_{k,\lambda}(2^s)&= \frac{1}{2^{2(k-1)}} \Omega(k,2,s,\frac{ r-4}{2})\cdot(\rho^{ }_{k,0}(8)-2^{2k-1})+\\
&\ \  +2^{1 + r - s + k\,\left( -2 - \frac{r}{2} + s \right) }  (\rho^{ }_{k,4\lambda'}(8)-2^{2k-1})+
2^{k \frac{r}{2} }\rho^{ }_{k,\lambda'}(2).
     \end{split} \end{equation*}
\end{itemize}
\end{cor}
\begin{proof}
\begin{itemize}
\item[i)] If $r$ is odd and $s-r>1$, then $r-2 i\geq3$ for every $i\leq \lfloor \frac{r}{2}-1\rfloor$. Consequently,
$$\rho^{(1)}_{k,\lambda/2^{2i}}(8)=\rho^{(1)}_{k,0}(8)=\rho^{ }_{k,0}(8)-2^{2 k-1}$$
due to Lemma \ref{L1} and
$$\sum_{i=0}^{\lfloor\frac{r}{2}\rfloor-1}\Big(2^{ki}2^{(s-2i-3)(k-1)}\cdot\rho^{(1)}_{k,\lambda/2^{2i}}(8)\Big)=\frac{1}{2^{2(k-1)}} \Omega(k,2,s,r,\frac{ r-3}{2})\cdot(\rho^{ }_{k,0}(8)-2^{2k-1}).$$
Finally, since
$$2^{k\lfloor\frac{r}{2}\rfloor}\rho^{(1)}_{k,\lambda/2^{2\lfloor\frac{r}{2}\rfloor}}(2^{s-2\lfloor\frac{r}{2}\rfloor})=2^{k\lfloor\frac{r}{2}\rfloor}\rho^{(1)}_{k,2\lambda' }(2^{s-r+1})=2^{k \frac{r-1}{2}+(s-r-2)(k-1) }\rho^{}_{k,2\lambda'}(8)$$
the result follows in this case.
\item[ii)] If $r$ is odd and $s-r=1$, we proceed like in the previous case but now we have that
$$2^{k\lfloor\frac{r}{2}\rfloor}\rho^{(1)}_{k,2\lambda' }(2^{s-r+1})=2^{k\lfloor\frac{r}{2}\rfloor}\rho^{(1)}_{k,2\lambda' }(4)=2^{k\lfloor\frac{r}{2}\rfloor}\rho^{ }_{k,2\lambda' }(4).$$
\item[iii)] If $r$ is even and $s-r>2$, then $r-2 i\geq3$ for every $i<\lfloor \frac{r }{2}-1\rfloor$, while $r-2i=2$ for $i=\lfloor \frac{r}{2}-1\rfloor$.
Thus,
$$ \sum_{i=0}^{\lfloor
\frac{r}{2}\rfloor-1}\Big(2^{ki}2^{(s-2i-3)(k-1)}\cdot\rho^{(1)}_{k,\lambda/2^{2i}}(8)
\Big)=$$
$$= \sum_{i=0}^{
\frac{r-4}{2} }\Big(2^{ki}2^{(s-2i-3)(k-1)}\cdot\rho^{(1)}_{k,0}(8)
\Big)   +2^{1 + r - s + k\,\left( -2 - \frac{r}{2} + s \right) }  (\rho^{1 }_{k,4\lambda'}(8))
=$$
$$= \sum_{i=0}^{
\frac{r-4}{2} }\Big(2^{ki}2^{(s-2i-3)(k-1)}\cdot\rho^{ }_{k,0}(8)
\Big)   +2^{1 + r - s + k\,\left( -2 - \frac{r}{2} + s \right) }  (\rho^{ }_{k,4\lambda'}(8)-2^{2k-1}).$$
\item[iv)] and v) If $r$ is even and $1\leq s-r\leq 2$, we proceed like in the previous case but now we have that
$$ 2^{k\lfloor\frac{r}{2}\rfloor}\rho^{(1)}_{k,\lambda/2^{2\lfloor\frac{r}{2}\rfloor}}(2^{s-2\lfloor\frac{r}{2}\rfloor})=2^{k \frac{r}{2} }\rho^{(1)}_{k,\lambda'}(2^{s-r})=2^{k \frac{r}{2} }\rho^{}_{k,\lambda'}(2^{s-r}).$$
\end{itemize}
\end{proof}

\begin{cor}
Let $2^s$ be a power of two ($s\geq 3$). Then,
$$\rho_{k,0}(2^s)=\begin{cases}\frac{1}{2^{2(k-1)}}\Omega(k,2,s,\frac{ s-3}{2}) \Big(\rho_{k,0}(8)-2^{2k-1}\Big) + 2^{ \frac{s-1}{2}k}\cdot 2^{k-1}+2^{\frac{s-1}{2}}, & \textrm{if $r$ is odd };\\ \frac{1}{2^{2(k-1)}}\Omega(k,2,s,\frac{ s-4}{2}) \Big(\rho_{k,0}(8)-2^{2k-1}\Big) + 2^{ \frac{s-2}{2}k}\cdot(\rho_{k,0}(4)-2^k)+2^{ \frac{s}{2}  k}, & \textrm{if $r$ is even }.\end{cases}$$
\end{cor}
\begin{proof}
For every $i\leq \lfloor \frac{s-1}{2}\rfloor-1$ we have that
$$\rho^{(1)}_{k,2^{s-2i}}(8)=\rho^{(1)}_{k,0}(8)=\rho^{ }_{k,0}(8)-2^{2 k-1}.$$
Now, if $r$ is odd
$$\rho^{(1)}_{k,0}(2^{s-2\lfloor\frac{s-1}{2}\rfloor})=\rho^{(1)}_{k,0}(2)=2^{k-1}.$$
While, if $r$ is even
$$\rho^{(1)}_{k,0}(2^{s-2\lfloor\frac{s-1}{2}\rfloor}=\rho^{(1)}_{k,0}(4)=\rho^{ }_{k,0}(4)-2^{k}.$$
In any case, it suffices to apply Theorem \ref{2z}.
\end{proof}

\section{Computational complexity of the computation of $\rho_{k,\lambda}(p^s)$}

The use of formulas for $\rho_{k,\lambda}(p^s)$ like (\ref{gaus}) and (\ref{dos}), based in the use of Gauss sums is ineffective, even for moderate small values of the parameters. For instance, the computation of $\rho_{10,5^{100000}}(5^{1000000})$ using (\ref{dos}) requires more than $10^{18}$ arithmetic operations. With the formulas that we have presented in this paper, the number of required arithmetic operations is of constant order and the aforementioned value can be computed in a domestic PC almost instantly. However, the mentioned arithmetic operations involve powers of integers as well as the computation of Legendre symbols in order to obtain the value of $\rho_{k,\lambda  }(p )$ via Proposition \ref{LEB}. These operations, when considered bit-wise, have a computational cost that increases with the size of the inputs. This becomes apparent when the involved parameters are very big.

If we have a look at the formulas presented in the previous section, those operations whose computational cost dominates over the others are the computation of the power $p^{k s}$ and the computation of the Legendre symbol. Their computational bit-level complexity is, respectively, $O(M(\log(p)) \log(ks))$ and $O(M(\log(p)) \log\log(p))$, where $M(n)$ represents the computational complexity of the chosen multiplication algorithm (see \cite{ja}). This gives an idea of which is the influence of each parameter over the overall computational cost of our procedure, as well as of its limitations. In fact, this reveals that the influence of the parameters $k$ and $s$ are similar (logarithmic order complexity) and somewhat lower to that of the prime $p$ when considering, for instance, the Schönhage-Strassen  multiplication algorithm whose computational complexity for the product of two numbers of size $n$ is $M(n)=O(n \log(n) \log\log(n))$ (see \cite{mul}) or the Fürer's algorithm \cite{mul2}, which runs in time $O(n \log (n) 2^ {O(\log^{*}( n)})$.

\end{document}